\documentclass[a4paper, 11pt,reqno]{amsart}

\usepackage[T1]{fontenc}      
\usepackage[foot]{amsaddr}    

\usepackage[british]{babel}   
\usepackage{csquotes}         
\usepackage[babel]{microtype} 

\usepackage{mathtools}        
\usepackage{amssymb}          
\usepackage{amsthm}           
\usepackage{amstext}          
\usepackage{mleftright}       
\usepackage{gensymb}          
\usepackage[makeroom]{cancel} 
\usepackage{mathdots}         
\usepackage{mathrsfs}         
\usepackage{dsfont}           
\usepackage[shortlabels]{enumitem}
\usepackage{verbatim}

\usepackage{stickstootext}
\usepackage[stickstoo,varbb]{newtxmath} 
\makeatletter
\DeclareFontFamily{U}{ntxmia}{\skewchar \font =127}
 \DeclareFontShape{U}{ntxmia}{m}{it}{
                        <-> \ntxmath@scaled ntxmia
                      }{}    
                      \DeclareFontShape{U}{ntxmia}{b}{it}{
                        <-> \ntxmath@scaled ntxbmia
                      }{}
\makeatother

\usepackage{graphicx}         
\usepackage{psfrag}           
\usepackage{caption}          
\usepackage{tikz}             
\usetikzlibrary{backgrounds}

\usepackage{booktabs}         

\usepackage{enumitem}         

\usepackage[margin=1in]{geometry} 

\usepackage[textsize=footnotesize]{todonotes}

\usepackage[numbers,sort]{natbib}

\makeatletter
\def\NAT@spacechar{~}
\makeatother
\newcommand{\urlprefix}{}

\usepackage{hyperref}              
\usepackage[capitalise]{cleveref}  
\hypersetup{colorlinks=true,
   citecolor=blue,
   filecolor=blue,
   linkcolor=blue,
   urlcolor=blue
}
\usepackage{url}

\makeatletter
\if@cref@capitalise
\crefname{figure}{figure}{figures}
\crefname{claim}{Claim}{Claims}
\else
\crefname{figure}{Figure}{Figures}
\crefname{claim}{claim}{claims}
\fi
\makeatother
\Crefname{figure}{Figure}{Figures}
\Crefname{claim}{Claim}{Claims}

\usepackage{algorithm}
\usepackage{algpseudocode}

\allowdisplaybreaks

\theoremstyle{definition}
\newtheorem{definition}{Definition}

\theoremstyle{plain}
\newtheorem{claim}{Claim}

\newtheorem{proposition}[definition]{Proposition}
\newtheorem{theorem}[definition]{Theorem}
\newtheorem{corollary}[definition]{Corollary}
\newtheorem{lemma}[definition]{Lemma}

\newtheorem{conjecture}[definition]{Conjecture}
\newtheorem{observation}[definition]{Observation}

\numberwithin{equation}{section}

\renewcommand{\binom}[2]{\ensuremath{\mleft(\kern-.1em\genfrac{}{}{0pt}{}{#1}{#2}\kern-.1em\mright)}}    
\newcommand{\inbinom}[2]{\ensuremath{\bigl(\kern-.1em\genfrac{}{}{0pt}{}{#1}{#2}\kern-.1em\bigr)}} 

\DeclareMathOperator\supp{supp}

\newcommand{\cA}{\mathcal{A}}

\newcommand{\cF}{\mathcal{F}}

\newcommand{\PP}{\mathbb{P}}

\newcommand{\EE}{\mathbb{E}}

\newcommand{\one}{\mathds{1}}
\newcommand{\lgr}{\left<g\right>}

\newcommand{\lGr}{\left<G\right>}

\newcommand{\lgrJL}{\left<g\right>_{J,L}}

\newcommand{\tf}{\tilde{f}}
\newcommand{\tg}{\tilde{g}}
\newcommand{\tC}{\tilde{C}}

\newcommand{\NN}{\mathbb{N}}

\makeatletter
\def\moverlay{\mathpalette\mov@rlay}
\def\mov@rlay#1#2{\leavevmode\vtop{%
  \baselineskip\z@skip \lineskiplimit-\maxdimen
  \ialign{\hfil$\m@th#1##$\hfil\cr#2\crcr}}}
\newcommand{\charfusion}[3][\mathord]{
    #1{\ifx#1\mathop\vphantom{#2}\fi
        \mathpalette\mov@rlay{#2\cr#3}
      }
    \ifx#1\mathop\expandafter\displaylimits\fi}
\makeatother

\renewcommand{\deg}{\operatorname{deg}}

\newcommand{\COMMENT}[1]{}
\renewcommand{\COMMENT}[1]{\footnote{\textcolor{blue!70!black}{#1}}} 
\newcommand{\COMNEW}[1]{}
\renewcommand{\COMNEW}[1]{\footnote{\textcolor{red!70!black}{#1}}} 

\title{Some results on fractional vs. expectation thresholds}

\author[T.~Fischer]{Thomas Fischer}
\author[Y.~Person]{Yury Person}

\thanks{Research is supported by DFG grant PE 2299/3-1}

\date{\today}

\begin{document}

\begin{abstract}
A conjecture of Talagrand (2010) states that the so-called expectation and fractional expectation thresholds are always within at most some constant factor from each other. Expectation (resp. fractional expectation) threshold $q$ (resp. $q_f$) for an increasing nontrivial class $\mathcal{F}\subseteq 2^X$ allows to locate the threshold for $\mathcal{F}$ within a logarithmic factor (these are important breakthrough results of Park and Pham (2022), resp. Frankston, Kahn, Narayanan and Park (2019)). We will survey what is known about the relation between $q$ and $q_f$ and prove some further special cases of Talagrand’s conjecture.
\end{abstract}

\maketitle

\section{Introduction}
Let $X$ be a finite nonempty set and let $p\in[0,1]$. An abstract model for studying random subsets of $X$ is often denoted by $X_p$, where each element from $X$ is included in $X_p$ independently of the others with probability $p$. Depending on the choice of $X$, one obtains various probabilistic models which were studied extensively in the last decades, such as a random subset of the positive integers $[n]_p$ (with $X=[n]$), binomial random graph $G(n,p)$ (with $X=\binom{[n]}{2}$ and we identify a graph with its edges), random $k$-uniform hypergraph $H^{(k)}(n,p)$ (where $X=\binom{[n]}{k}$). For initial pointers to the literature we refer to standard reference books such as~\cite{AS2016, Bol01, JLR00, FK16}.

 For a given set $X$ we denote its power set by $2^X$. A set (or a class) $\cF\subseteq 2^X$ is called a \emph{property}; moreover, we say that $\cF$ is \emph{nontrivial} if $\cF\neq \emptyset$, $2^X$. A property $\cF$ is called \emph{increasing} if  whenever $A\in \cF$ and  $A\subseteq B$ we have $B\in \cF$ (that is, adding elements to a set $A\in \cF$ does not destroy the property). An example of an increasing graph property  is  subgraph containment such as, say, hamiltonicity. 

For any increasing nontrivial $\cF\subseteq 2^X$ we will be interested in the probability $\PP[X_p\in\cF]$, that is, how likely is $X_p$ to possess the property $\cF$? One can show that the function $f(p):=\PP[X_p\in\cF]$ is continuous and strictly increasing in $p$ and hence there exists a unique $p_c\in[0,1]$ with $\PP[X_{p_c}\in\cF]=1/2$. Such value $p_c$ is called \emph{the threshold} for $\cF$. 

In the setting of random graphs, threshold functions were discovered by Erd\H{o}s and R\'enyi in~\cite{ErdRen60} who observed that many graph properties possess thresholds. The study of thresholds has been and remains one of the central topics of study in the theory of random graphs ever since. As for a general result,  Bollob\'as and Thomason~\cite{BolTho87} proved that every nontrivial increasing property $\cF$ has a threshold function.\footnote{A threshold function for properties of random graphs is defined somewhat differently, as follows: $\hat{p}\colon \NN\to [0,1]$ is a threshold for some property $\cA=\cup_{n\in\NN}\cA_n$, if $\PP[G(n,p)\in\cA_n]\to 1$ for $p=\omega(\hat{p})$ and $\PP[G(n,p)\in\cA_n]\to 0$ for $p=o(\hat{p})$ as $n$ tends to infinity (observe that one considers a sequence of probability spaces). Defining for a nontrivial increasing property $\cA_n$ the threshold $\hat{p}(n):=p_c$ with $\PP[G(n,p_c)\in\cA_n]=1/2$ for every $n$, the function $\hat{p}$ is a threshold function.} Despite the fact of knowing that a threshold function exists, it is not clear at all how to determine it. Kahn and Kalai~\cite{KK07} and subsequently Talagrand~\cite{TalagrandOriginal} proposed far-reaching conjectures about the location of the threshold $p_c$. The latter was proved by Frankston, Kahn, Narayanan and Park~\cite{FKNP19} (following ideas from a breakthrough of \citet{ALWZ21} on the sunflower conjecture) while the former was proved by Park and Pham in~\cite{PP23}. In the following we introduce some notation in order to describe the results and come to the matter of the present paper.

 One common idea for a lower bound on $p_c$ is to find a random variable $Y\geq 0$ (dependent on $p$ and $n$) such that $Y\geq 1$ whenever $X_{p}\in\mathcal{F}$  holds. This yields with Markov's inequality that $\mathbb{P}\left(X_{p}\in\mathcal{F}\right)\leq \mathbb{P}\left(Y\geq 1\right)\leq \mathbb{E}\left[Y\right]$. 
 
 One possibility to construct such a random variable $Y$ is to find a set $G\subseteq 2^X$ such that for every $S\in\mathcal{F}$ there exists a $T\in G$ such that $T\subseteq S$ and then let $Y=Y_G$ denote the number of $T\in G$ which are contained in $X_p$, i.e.\ $T\subseteq X_p$. It should be clear that $\EE\left[Y_G\right]=\sum_{T\in G} p^{\left|T\right|}$. We call $\cF$ to be \emph{$p$-small} if there exists such a set $G\subseteq 2^X$ such that $\EE[Y_G]\le 1/2$. 

A somewhat advanced way to construct a random variable $Y$ is  to find a function $g\colon 2^X\rightarrow\left[0,1\right]$ which satisfies for every $S\in\cF$ the inequality $\sum_{T\subseteq S}g(T)\geq 1$ and to choose $Y=Y_g:=\sum_{T\subseteq X_p} g(T)$. In this case the function $g$ can be thought of as a fractional version for the set $G$ above (it should be clear that $Y_G=Y_{\one_G}$). Again we can easily compute $\mathbb{E}\left[Y_g\right]=\sum_{T\in 2^X} g(T)\cdot p^{\left|T\right|}$.  We call $\cF$ to be \emph{weakly $p$-small} if we can find such a function $g$ as above (i.e.\ $X_{p}\in\mathcal{F}\Rightarrow Y_g\geq 1$) with  $\mathbb{E}\left[Y_g\right]\leq \frac{1}{2}$. It should be clear that if $\cF$ is $p$-small then $\cF$ is also weakly $p$-small with $g=\one_G$.

\begin{definition}\label{def:p-small}
For an increasing nontrivial $\cF\subseteq 2^X$ define:
\begin{equation}\label{eq:exp-thr}
q:=q(\cF):=\max\{p\colon \cF \text{ is $p$-small}\},
\end{equation}
and 
\begin{equation}\label{eq:exp-thr-w}
q_f:=q_f(\cF):=\max\{p\colon \cF \text{ is weakly $p$-small}\}.
\end{equation}
One refers to $q$ as the \emph{expectation threshold}, whereas $q_f$ is the \emph{fractional expectation threshold}.
\end{definition}

From the discussion above we see immediately that $q\le q_f\le p_c$ (by using $g=\mathds{1}_{G}$). Kahn and Kalai conjectured in~\cite{KK07}  that there exists a universal constant $K>0$ such that $p_c\le K q \log |X|$, and  Talagrand conjectured in~\cite{TalagrandOriginal} an apparently weaker form: $p_c\le K  q_f \log |X|$. These conjectures were resolved in breakthrough works   in~\cite{PP23} and in~\cite{FKNP19}\footnote{In fact, Talagrand conjectured the strengthenings of both conjectures that $\log |X|$ can be replaced by $\log \ell$, where $\ell$ is the largest cardinality among all minimal sets of $\cF$. Both strengthenings were proved in~\cite{PP23} resp.\  in~\cite{FKNP19}.}, which allows for many properties to obtain new results or to provide alternative proofs of deep results from the theory of random graphs.  
 
Until the proof of the Kahn-Kalai conjecture in~\cite{PP23}, a promising route was suggested by Talagrand in~\cite[Conjecture 6.3]{TalagrandOriginal} to show that $q$ and $q_f$ are always within a constant factor of each other.
\begin{conjecture}[Talagrand~\cite{TalagrandOriginal}]\label{conj:smallness}
    There exists some fixed $L>1$ such that for every finite nonempty $X$ and any nontrivial $\mathcal{F}\subseteq 2^X$ the following is true: If $\mathcal{F}$ is weakly $p$-small then   $\cF$ is also $(p/L)$-small. Equivalently: $q_f\leq L\cdot q$.
\end{conjecture}

We find it convenient to formulate the problem somewhat differently. For this purpose we are introducing some additional notation. 

First we introduce the weight functions implicitly mentioned previously.
\begin{definition}\label{def:weight_fcts}
    For $p\in [0,1]$, $G\subseteq 2^X$ and $g: 2^X \rightarrow [0,\infty)$ we define:
    \begin{alignat*}{100}
        w(G,p) & := \sum_{T\in 2^X} & \mathds{1}_{\{T\in G\}} \cdot p^{|T|} \\
        w(g,p) & := \sum_{T\in 2^X} & g(T)\cdot p^{|T|}
    \end{alignat*}
\end{definition}
Thus, in the notation above the weights $w(G,p)$ and $w(g,p)$ correspond to $\EE[Y_G]$ and $\EE[Y_g]$. Moreover, the range of $g$ is allowed to be $[0,\infty)$ since it will be used in later sections. 

The set $G$ and the function $g$ which are needed for $\cF$ to be $p$- resp.\ weakly $p$-small imply that the following sets contain $\cF$ (corresponding to the sets $\{S\colon Y_G(S)\ge 1\}$ and $\{S\colon Y_g(S)\ge 1\}$ respectively).
\begin{definition}\label{def:cov_sets}
    For $G\subseteq 2^X$ and $g\colon 2^X \rightarrow [0,\infty)$ we define:
    \begin{alignat*}{100}
        \left<G\right> & := \left\{ S\in 2^X \left| \ \sum_{T\subseteq S} \mathds{1}_{\{T\in G\}} \geq 1 \right. \right\}  \\
        \left<g\right> & := \left\{ S\in 2^X \left| \  \sum_{T\subseteq S} g(T) \geq 1 \right. \right\}
    \end{alignat*}
\end{definition}

The following observation is straightforward.
\begin{observation}\label{obs:p-small}
 We have:
\begin{enumerate}[(i)]
\item 
The set $\cF$ is $p$-small if and only if there exists a set $G\subseteq 2^X$ satisfying $\cF\subseteq \left<G\right>$ and $w(G,p)\leq 1/2$.
\item 
The set $\cF$ is weakly $p$-small if and only if there exists a function $g\colon 2^X \rightarrow [0,1]$ satisfying $\cF\subseteq \left<g\right>$ and $w(g,p)\leq 1/2$.
\end{enumerate}
\end{observation}

Now as a simplification since $X$ is finite there is always an $n\in\NN$ with $|X|=n$ and we can think of $X$ as $[n]=\{1,\ldots, n\}$. Moreover, since $p^0=1$ we may assume that $\emptyset\not\in G$.

Recall that the largest $p$ in Observation~\ref{obs:p-small} corresponds to $q$ and $q_f$ resp. and $\emptyset\not\in G$, hence if we change $1/2$ to $1$ in Definition~\ref{eq:exp-thr} and require $g(\emptyset)=0$, then the values $q$ and $q_f$ change at most by  a factor of $2$. We thus restate Talagrand's conjecture (Conjecture~\ref{conj:smallness}) as follows.

\begin{conjecture}\label{TalagrandOurFormulation}\label{conj:Talagrand_b}
    There exists some fixed $L>1$ such that for all $n\in\NN$, $g: 2^X \rightarrow [0,1]$ with $g(\emptyset)=0$ and $p\in [0,1]$ the following holds. 
 If     $w(g,p) = 1$ then there exists a set $G\subseteq 2^X\setminus\{\emptyset\}$ with $\left<g\right>\subseteq \left<G\right>$ and $w\left(G,\frac{p}{L}\right) \leq 1$.   
\end{conjecture}
 It is not difficult to see that Conjectures~\ref{conj:smallness} and~\ref{conj:Talagrand_b} are equivalent. Indeed, if the former is true and $w(g,p) = 1$, then $\cF:=\left<g\right>$ 
 is weakly $p'$-small for some $p'\in[p/2,p]$ and hence $\cF$ is $p'/L$-small. Thus there exists a set $G\subseteq 2^X$ with $\cF=\left<g\right>\subseteq \left<G\right>$ and $w\left(G,\frac{p}{2L}\right)\le w\left(G,\frac{p'}{L}\right) \leq 1/2$. On the other hand, if Conjecture~\ref{conj:Talagrand_b} is true and some nontrivial set $\cF\subseteq 2^X$ is weakly $p$-small then there exists a  weight function $g\colon 2^X\to[0,1]$ with $\cF\subseteq \left<g\right>$ and $w(g,p)\le 1/2$. This implies $g(\emptyset)\le 1/2$. We define $\tilde{g}(\emptyset)=0$ and for $T\neq \emptyset$ we set $\tilde{g}(T):=\min\{2g(T),1\}$ and observe that $\lgr\subseteq \left <\tilde{g}\right>$ and $w(\tilde{g},p)\le 2\sum_{T\in 2^X} g(T)p^{|T|}\le 1$. We thus find $\tilde{p}\ge p$ so that $w(\tilde{g},\tilde{p})=1$. The truth of Conjecture~\ref{conj:Talagrand_b} implies the existence of $G\subseteq 2^X\setminus\{\emptyset\}$ with 
 $\cF\subseteq \left<g\right>\subseteq \left<\tilde{g}\right>\subseteq \left<G\right>$ and $w\left(G,\frac{\tilde{p}}{L}\right) \leq 1$. Hence we obtain $w\left(G,\frac{p}{2L}\right)\le w\left(G,\frac{\tilde{p}}{2L}\right) \leq 1/2$.

 Another equivalent version of Conjecture~\ref{conj:smallness} can be stated as follows. 
 \begin{conjecture}\label{conj:KKweak}
  There is a fixed $L>0$ such that for any finite set $X$, any $p\in[0,1]$ and function $\lambda\colon 2^X\setminus\{\emptyset\}\to [0,\infty)$ the set
  \[
  \left\{S\subset X\colon \sum_{T\subseteq S} \lambda(T)\ge \sum_{T\in 2^X\setminus\{\emptyset\}} (Lp)^{|T|}\lambda(T)\right\}
  \] 
  is $p$-small.
 \end{conjecture}
 For further details we refer the interested reader to~\cite{TalagrandCliques, TalagrandTwoSets}.

To the best of our knowledge, Conjecture~\ref{conj:Talagrand_b} is open, and only some special cases of it have been solved. 
Talagrand~\cite{TalagrandOriginal} proved Conjecture~\ref{conj:Talagrand_b} for functions $g$ whose support $\supp(g):=\{S\colon g(S)\neq 0\}$ is contained in $\binom{X}{1}$ and 
also for functions $g$ so that, for some set $J\subseteq X$, all sets $S$ from $\lgr$ contain at least $(2e)p|J|$ elements from $J$ (see~\cite[Lemma~5.9]{TalagrandOriginal}). Talagrand~\cite{TalagrandOriginal} also 
suggested two further special cases as test cases: when $g$ is constant and supported by the edge sets of the cliques of some fixed order $k$ in the complete graph $K_n$ and when $g$ is supported by a subset of $2$-sets in $X$ (i.e.\ $\supp(g)\subseteq \binom{X}{2}$). The former case was verified by DeMarco and Kahn in~\cite{TalagrandCliques} and the latter was verified by Frankston, Kahn and Park in~\cite{TalagrandTwoSets}. 

 From now on we will always assume that $g(\emptyset)=0$ and that $\emptyset\not\in G$ often without explicitly mentioning it. In the remainder of the Introduction we state more results about Conjecture~\ref{conj:Talagrand_b} alongside some remarks. All this should help to structure the conjecture into easier to handle cases. We provide the proofs in the subsequent sections. 

It will be useful to think of the function $g$ as corresponding to a weighted hypergraph on $X$ (where the weighted edges correspond to the sets from $\supp(g)$).

Our first result allows to reduce the problem to the case when the function $g$ has its support in $\binom{X}{k}$ for some $k\in \NN$.
\begin{theorem}\label{thm:reduction_uniform}
Suppose there exists some $L>1$ such that for all $k\in\NN$ and all finite sets $X$ the following holds. 
Whenever a function $g_k\colon \binom{X}{k} \rightarrow [0,1]$ satisfies $w(g_k,p) \le 1$ for some $p\in[0,1]$, there exists a set 
$G_k\subseteq 2^X$ with $\left<g_k\right>\subseteq \left<G_k\right>$ and  $w\left(G_k,\frac{p}{L}\right) \leq 1$.

 Then the following is true for any finite set $X$. 
If a function $g\colon 2^X \rightarrow [0,1]$ satisfies $w(g,p) = 1$ then there exists a set  
$G\subseteq 2^X$ with $\left<g\right>\subseteq \left<G\right>$ and $w\left(G,\frac{p}{4L}\right) \leq 1$.
\end{theorem}

We remark that the assumption $w(g_k,p) \le 1$ above could be replaced by the only apparently  stronger $w(g_k,p) = 1$, since, by monotonicity of the weight function we could pick some $p'$ with 
$w(g_k,p') = 1$ and from the monotonicity of $w\left(G_k,p'/L\right)$ we would obtain $w\left(G_k,p/L\right)\le w\left(G_k,p'/L\right)\le 1$. 
The above reduction allows to work in the uniform setting to attack Conjecture~\ref{conj:Talagrand_b} and simplifies it to the following conjecture:

\begin{conjecture}\label{TalagrandOurFormulation_uniform}\label{conj:Talagrand_c}
    There exists some fixed $L>1$ such that for all finites sets $X$ and all $k\in\NN$, $g: \binom{X}{k} \rightarrow [0,1]$ with $p\in [0,1]$ the following holds. 
 If     $w(g,p) = 1$ then there exists a set $G\subseteq 2^X\setminus\{\emptyset\}$ with $\left<g\right>\subseteq \left<G\right>$ and $w\left(G,\frac{p}{L}\right) \leq 1$.   
\end{conjecture}

In particular this allows us to rewrite $w(g,p)=1$ as 
$p = \left(\sum_{S\in \binom{X}{k}} g_k\left(S\right)\right)^{-\frac{1}{k}}$ (cf.\ Definition~\ref{def:weight_fcts}).

A first natural approach to tackle Conjecture~\ref{conj:Talagrand_c} is to `(de-)randomize' the function $g$ by interpreting its weights (or multiples thereof) as probabilities and trying to construct a set $G$ with $G\subseteq \supp(g)$. Our next result establishes the following.

\begin{theorem}\label{thm:randomized}
Let $p\in[0,1]$, $n\ge k\in\NN$, $X$ be an $n$-element set and $g\colon \binom{X}{k} \rightarrow [0,1]$ a function. If $ w(g,p) = 1$ then there exists $G\subseteq \supp(g)$ with $\left<g\right>\subseteq \left<G\right>$  and  $w\left(G,\frac{p}{4\cdot n^{\frac{1}{k}}}\right) \leq 1$. In particular, for $k\ge C \log_2(n)$ with some $C>0$, it holds that if $ w(g,p) = 1$ then there exists $G\subseteq \supp(g)$ with $\left<g\right>\subseteq \left<G\right>$  and $w\left(G,\frac{p}{4\cdot 2^{1/C}}\right) \leq 1$.
\end{theorem}

We remark that the first statement in Theorem~\ref{thm:randomized} is asymptotically optimal in the sense that the bound $p/n^{\frac{1}{k}}$ cannot be substantially improved as the following example certifies. Let $n$ be even and define $g=\binom{n/2}{k}^{-1}\cdot \one_{\binom{X}{k}}$. Then $\lgr=\{S\colon |S|\ge n/2\}$ and from $w(g,p)=1$ it follows that $p\approx 1/2$, while $G\subseteq \supp(g)=\binom{X}{k}$ and $\lgr\subseteq\lGr$ imply that $|G|\ge n/(2k)$. Thus, if  $w(G,p/L)=|G|(p/L)^k\le 1$ then $L=\Omega(n^{1/k})$.

Additionally we remark that the second statement of Theorem~\ref{thm:randomized} can also be proven with the methods used in \cite{TalagrandCliques} where it is proven for the clique case.
Furthermore we remark that it is possible with some additional care to improve the first result to $w\left(G,\frac{p}{4\cdot \log_2\left(m\right)^{\frac{1}{k}}}\right) \leq 1$ for $m$ being the number of minimal elements in $\lgr$, which also improves the second result to $k\ge C \log_2(\log_2\left(m\right))$.

The following proposition deals with the case when $\lgr$ only contains large enough sets. Its first part was proved by Talagrand~\cite[Lemma~5.9]{TalagrandOriginal}. 
For the sake of completeness we provide the proof of the first part as well.
\begin{proposition}\label{prop:volume}
Let $n$, $k\in\NN$ with $n\ge k$, let $p\in[0,1]$, let $V\subseteq X$, where $X$ is an $n$-element set and $V\neq \emptyset$. If $g\colon \binom{X}{k}\to[0,1]$ is a function with $w(g,p)=1$  such that  every $S\in\lgr$ satisfies $|S\cap V|\ge \frac{ep}{L}|V|$ (for some $L>0$) then there exists $G\subseteq 2^X$ with $\lgr\subseteq\lGr$ and $w(G,\frac{p}{L})\le 1$. 

In particular, if we have a function $g\colon \binom{V}{k}\to[0,1]$ which is constant on its support, satisfying $w(g,p)=1$ and $|\supp(g)|\ge \left(\frac{e^2}{L}\right)^k\binom{|V|}{k}$, then there exists $G\subseteq 2^X$ with $\lgr\subseteq\lGr$ and $w(G,\frac{p}{L})\le 1$. 
\end{proposition}

\begin{proof}[Proof of Proposition~\ref{prop:volume}]
    By taking 
    \begin{alignat*}{100}
        G=\binom{V}{\left\lceil\frac{e\cdot p}{L}\cdot \left|V\right|\right\rceil}
    \end{alignat*}
    we get $\left<g\right>\subseteq \left<G\right>$ since all $S\in\left<g\right>$ intersect $V$ in at least $\left\lceil\frac{e\cdot p}{L}\cdot \left|V\right|\right\rceil$ elements and we obtain 
    \begin{alignat*}{100}
        w\left(G,\frac{p}{L}\right)= \binom{\left|V\right|}{\left\lceil\frac{e\cdot p}{L}\cdot \left|V\right|\right\rceil}\cdot \left(\frac{p}{L}\right)^{\left\lceil\frac{e\cdot p}{L}\cdot \left|V\right|\right\rceil}\leq \left(\frac{e\cdot \left|V\right|\cdot p}{{\left\lceil\frac{e\cdot p}{L}\cdot \left|V\right|\right\rceil}\cdot L}\right)^{\left\lceil\frac{e\cdot p}{L}\cdot \left|V\right|\right\rceil} \leq \left(\frac{e\cdot \left|V\right|\cdot p}{{\frac{e\cdot p}{L}\cdot \left|V\right|}\cdot L}\right)^{\left\lceil\frac{e\cdot p}{L}\cdot \left|V\right|\right\rceil} = 1.
    \end{alignat*}

    Now let $g\colon \binom{V}{k}\to[0,1]$ be constant on its support, i.e.\ $g(T)=c$ for all $T\in\supp(g)$ for some $c>0$ and satisfy  $w(g,p)=1$ and $|\supp(g)|\ge \left(\frac{e^2}{L}\right)^k\binom{|V|}{k}$. 
    Let $s\in\NN$, $s\ge k$,  be such that $\binom{s}{k}^{-1}\le c< \binom{s-1}{k}^{-1}$. Then $S\in\lgr$ has at least $s$ elements in $V$. From $w(g,p)=1$ we obtain
    \begin{alignat*}{100}
        p=\left(|\supp(g)|c\right)^{-1/k}\le \frac{\binom{s}{k}^{1/k}}{\left(\left(\frac{e^2}{L}\right)^k\binom{|V|}{k}\right)^{1/k}}\le 
        \frac{(es/k)}{(e^2/L)\cdot (|V|/k)}=\frac{Ls}{e|V|},
    \end{alignat*}
    from which $\frac{ep}{L}|V|\le s$ follows. Hence, we can apply the first statement of the proposition and get the conclusion.
\end{proof}

 A $k$-uniform hypergraph $H\subseteq \binom{X}{k}$ on the vertex set $X$ is called \emph{linear}, if any two distinct vertices lie in at most one (hyper-)edge. For a set $V\subseteq X$, we denote through $H[V]$ the \emph{induced} hypergraph $H\cap\binom{V}{k}$ on $V$.
  For two distinct vertices $x$, $y\in X$, the \emph{codegree of $x$ and $y$ in $H$} will be denoted by $\deg_H(x,y)$ (or simply $\deg(x,y)$ when $H$ is clear from the context) and defined through $\deg(x,y):=|\{e\colon e\in H \text{ and } e\ni x, y\}|$. The \emph{maximum codegree of $H$} is  $\Delta_2(H):=\max_{\{x,y\}\in\binom{X}{2}}\deg_H(x,y)$. Our last result generalizes the result of Frankston, Kahn and Park~\cite{TalagrandTwoSets} when $\supp(g)\subseteq \binom{X}{2}$ to the case when $H:=\supp(g)$ is a `nearly' linear $k$-uniform hypergraph.
\begin{theorem}\label{TheoremSimpleHypergraphs}
There exists some constant $C>1$ such that the following holds. Let $X$ be a finite set and let  $k$, $c\in\NN$ with $|X|\ge k$, $p\in[0,1]$ and  let $g\colon \binom{X}{k} \rightarrow [0,1]$ be a function such that 
    \begin{alignat*}{100}
       \Delta_2(\supp(g))\le c^k.
    \end{alignat*}     
If $w(g,p) = 1$ then there exists $G\subseteq 2^X$ with $\lgr\subseteq\lGr$ and 
\[
w\left(G,\frac{p}{C\cdot c^2}\right) \leq 1.
\]
\end{theorem}

As a corollary we obtain that Conjecture~\ref{conj:Talagrand_b} holds when $g$ is supported only by $k$-APs, where a $k$-AP stands for a $k$-term arithmetic progression, i.e.\ a set of the form $\{a,a+d,\ldots,a+(k-1)d\}$ (with $d\neq0$). Indeed, this follows since the maximum codegree of the $k$-uniform hypergraph of $k$-APs is at most $(k-1)^2$.

\begin{corollary}\label{cor:k-APs}
There exists a constant $C>1$ such that for all $n$, $k\in \NN$ the following holds.
If $X=[n]$ and a function $g\colon \binom{X}{k}\to [0,1]$ is such that $\supp(g)$ consists only of $k$-APs and $w(g,p)=1$ for some $p\in[0,1]$, then $\lgr$ is $p/C$-small.
\end{corollary}

In the subsequent sections we provide proofs of Theorems~\ref{thm:reduction_uniform},~\ref{thm:randomized} and~\ref{TheoremSimpleHypergraphs} respectively.

\section{Proof of Theorem~\ref{thm:reduction_uniform}}
The idea behind the proof of Theorem~\ref{thm:reduction_uniform} is to decompose the given function $g\colon 2^X\to[0,1]$ into functions $g_k\colon \binom{X}{k}\to[0,1]$. Knowing that $\left<g_k\right>$ are contained in $\left< G_k\right>$, we will need to establish the relation between $\lgr$ and $\left<\cup_{k=1}^n G_k\right>$. We will also need an auxiliary lemma, Lemma~\ref{lem:PowerTrick} below, which allows us to deduce that, under quite natural conditions, we can find $G_k$ so that $\left< g_k\right> \subseteq \left<G_k\right>$ and $w(G_k,\frac{p}{c L})\le \frac{1}{c^k}$, which will help us to bound the weight of $\cup_{k=1}^n G_k$.

 We start with the following definition.
\begin{definition}\label{def:Xm} 
    For a set $X$, numbers $k$, $m\in\NN$ with $|X|\ge k$ and a function $g\colon \binom{X}{k}\rightarrow [0,\infty)$, let $\left(X_i\right)_{i=1}^m$ be distinguishable and disjoint copies of $X$ and $g_i\colon \binom{X_i}{k} \rightarrow [0,\infty)$ be copies of $g$.
       We define:
    \begin{alignat*}{100}
        X^{(m)} & := \dot\bigcup_{i=1}^m X_i, \\
        g^{(m)} & \colon \binom{X^{(m)}}{k}\rightarrow [0,\infty),\\
        g^{(m)}(S)& := \sum_{i=1}^m \mathds{1}_{\{S \subseteq X_i\}} g_i(S).
    \end{alignat*}
\end{definition}

The following lemma allows us to infer that, under quite natural conditions, the weight of a `covering' set $G$ `scales' with $p/c$ at the `rate' $k$ even though $G$ need not consist only of sets of cardinality at least $k$.

\begin{lemma}\label{lem:PowerTrick}
Let $p\in[0,1]$, $L>0$, $c$, $k\in\NN$, let $X$ be a finite set with $|X|\ge k$ and $g\colon \binom{X}{k}\to[0,\infty)$ a function. 
Set $m=c^k$ and suppose that $w\left(g^{(m)},\frac{p}{c}\right)\le 1$ and there exists a set $G^{(m)}\subseteq 2^{X^{(m)}}$ such that $\left<g^{(m)}\right> \subseteq \left<G^{(m)}\right>$ and 
$w\left(G^{(m)},\frac{p}{c\cdot L}\right)\le 1$.  Then $w\left(g,p\right)\le 1 $ and there exists a set  $G\subseteq 2^X$  with $\left<g\right>\subseteq \left<G\right>$ and  $w\left(G,\frac{p}{c\cdot L}\right) \leq \frac{1}{c^k}$.
\end{lemma}

\begin{proof} We have
    \begin{alignat*}{100}
         w\left(g,p\right)= m\cdot w\left(g,\frac{p}{c}\right)  =  w\left(g^{(m)},\frac{p}{c}\right)\le 1,
    \end{alignat*}
 where  the first equality follows from the definition of $w(g,p)$, $m=c^k$ and since $\supp(g)\subseteq\binom{X}{k}$ (cf.\ Definition~\ref{def:weight_fcts}), while the second equality holds since $g_i$'s are copies of $g$.

  Observe that all $\left<g_i\right>\subseteq 2^{X_i}$ ($i\in[m]$) are  disjoint and contained in $\left<g^{(m)}\right>$. 
Hence, there exist
   pairwise disjoint $G_i\subseteq 2^{X_i}\cap G^{(m)}$ with  $\left<g_i\right>\subseteq \left<G_i\right>$ for every $i\in[m]$.  Therefore it follows 
    \begin{alignat*}{100}
        \sum_{i=1}^m w\left(G_i,\frac{p}{c\cdot L}\right) \leq w\left(G^{(m)},\frac{p}{c\cdot L}\right) \leq 1. 
    \end{alignat*}
Thus there exists an $i\in[m]$ such that $w\left(G_i,\frac{p}{c\cdot L}\right)\leq \frac{1}{m}=\frac{1}{c^k}$. 
Since $\left<g_i\right>$ is a copy of $\left<g\right>$ there also exists a required $G$.
\end{proof}

Now we are in position to prove Theorem~\ref{thm:reduction_uniform}.

\begin{proof}[Proof of Theorem~\ref{thm:reduction_uniform}]
Let $X$ be any fixed finite set and let $g\colon 2^X\to[0,1]$ be any function with $w(g,p)=1$. We can write $g=\sum_{k=1}^n g_k$,  where 
$g_k\colon \binom{X}{k} \rightarrow [0,1]$  are defined through $g_k(T):=g(T)\cdot \one_{\left\{T \in \binom{X}{k}\right\}}$ for all $T\in\binom{X}{k}$.

Next we define, for every $k\in[n]$, a function 
 $h_k\colon  \binom{X}{k} \to [0,1]$ by setting $h_k(T):=\min\left\{2^k\cdot g_k(T),1\right\}$ for all $T\in \binom{X}{k}$. We then have  that        
    \begin{alignat*}{100}
w\left(h_k,\frac{p}{2}\right)\le 2^k\cdot w\left(g_k,\frac{p}{2}\right) = \frac{2^k}{2^k}\cdot w(g_k,p) \leq w(g,p) =1,
    \end{alignat*}
where  we use the definition of $w(g,p)$ (cf.\ Definition~\ref{def:weight_fcts}),  $h_k\le 2^k\cdot g_k$ and the assumption of the theorem that $w(g,p) = 1$.

 We choose $c=2$ and set $m=c^k$. By construction (cf.\ Definition~\ref{def:Xm}), we have $h_k^{(m)}\colon \binom{X^{m}}{k}\to[0,1]$ and 
$w\left(h_k^{(m)},\frac{p}{4}\right)=w\left(h_k^{(m)},\frac{p/2}{2}\right)=\frac{m}{2^k}\left(h_k,p/2\right)=\left(h_k,p/2\right)\le 1$.  In the following we assume w.l.o.g.\ that, for every $k\in[n]$, $w(h_k,1)\ge 1$ (as otherwise $\left<h_k\right>=\emptyset$ and we can set $G_k=\emptyset$ in what follows). 
The assumption of Theorem~\ref{thm:reduction_uniform} on $h_k^{(m)}$ (as $g_k$) and $p$ (as $p/4$) makes Lemma~\ref{lem:PowerTrick} applicable to $h_k$ (as $g$), $p/2$ (as $p$) and $c=2$. 
Hence, there exists a set $G_k\subseteq 2^X$ with $\left<h_k\right>\subseteq \left<G_k\right>$ and  $w\left(G_k,\frac{p}{4L}\right) =w\left(G_k,\frac{p/2}{2L}\right) \leq \frac{1}{2^k}$. 
    Therefore we obtain:
    \begin{alignat}{100}\label{eq:weight}
        w\left(\bigcup_{k=1}^n G_k,\frac{p}{4 L}\right)\leq \sum_{k=1}^n w\left(G_k,\frac{p}{4 L}\right)< \sum_{k=1}^\infty \frac{1}{2^k} = 1
    \end{alignat}
    On the other hand, for every $S\in \left<g\right>$ we have:
    \begin{alignat*}{100}
        \sum_{k=1}^n \sum_{T\subseteq S} g_k(T) = \sum_{T\subseteq S} g(T) \geq 1
    \end{alignat*}
   By the pigeonhole principle there is an index $k\in[n]$ such that $\sum_{T\subseteq S} g_k(T) \geq \frac{1}{2^k}$ and hence (since $h_k(T)=\min\left\{2^k\cdot g_k(T),1\right\} $) this also yields $\sum_{T\subseteq S} h_k(T) \geq 1$, which shows $S\in \left<h_k\right>$. 
   Therefore the following sequence of inclusions holds
    \begin{alignat}{100}\label{eq:cover}
        \left<g\right>\subseteq \bigcup_{k=1}^n\left<h_k\right>\subseteq \bigcup_{k=1}^n\left<G_k\right> = \left<\bigcup_{k=1}^n G_k\right>
    \end{alignat}
The claim now follows with  $G=\bigcup_{k=1}^n G_k$ and~\eqref{eq:weight},~\eqref{eq:cover}. 
\end{proof}

\section{Proof of Theorem~\ref{thm:randomized}}

We will construct a desired set $G\subseteq 2^X$ from $\supp(g)$ by taking each $S\in \lgr$ randomly with probability proportional to the weight $g(S)$ independently of the others. The following lemma estimates the probability that such random $G$ satisfies $\lgr\subseteq\lGr$.

\begin{lemma}
\label{LemmaProbCovering}
Let $n$, $k\in\NN$ with $k\le n$  and let $X$ be an $n$-element set. 
Let $c>0$ and let $g\colon \binom{X}{k}\to[0,1]$ be a function. 
The set $G\subseteq\supp(g)$ is obtained by including each $T\in\supp(g)$ into $G$ with probability $\PP(T\in G):=\min\left\{cg(T),1\right\}$ independently of the others. Then we have
    \begin{alignat*}{100}
        \mathbb{P}\left(\left<g\right>\subseteq \left<G\right>\right)\geq 1 - e^{n-c}
    \end{alignat*}
\end{lemma}
\begin{proof}
    Let $S\in\left<g\right>$. From Definition~\ref{def:cov_sets}  we know that $\sum_{T\subseteq S} g(T)\geq 1$ and we get (we assume $c\cdot g(T)<1$ for all $T\subseteq S$ otherwise this is trivial):
    \begin{alignat*}{100}
        \mathbb{P}\left(S\not\in \left<G\right>\right)= \prod_{T\subseteq S}\mathbb{P}\left(T\not\in G\right)= \prod_{T\subseteq S}\left(1-c\cdot g(T)\right)\leq \prod_{T\subseteq S}e^{-c\cdot g(T)}= e^{-c\cdot \sum_{T\subseteq S}g(T)}\leq e^{-c}
    \end{alignat*}
Taking union bound over all $S\in\left<g\right>$ and since $|\left<g\right>|\leq 2^n$ we obtain
    \begin{alignat*}{100}
        \mathbb{P}\left(\left<g\right>\subseteq \left<G\right>\right)= 1 - \mathbb{P}\left(\exists S\in\left<g\right>: S\not\in \left<G\right>\right) \geq 1 - 2^n\cdot e^{-c} \geq 1 - e^{n-c}.
    \end{alignat*}
\end{proof}

Next we estimate the expected weight of $G$ constructed in Lemma~\ref{LemmaProbCovering}.
\begin{lemma}
\label{LemmaExpectedWeight}
Let $g$ be a function and let $G$ be a random set as in assumption of Lemma~\ref{LemmaProbCovering}. 
If $w(g,p)=1$ we obtain for any $L\ge1$ that 
    \begin{alignat*}{100}
        \mathbb{E}\left[w\left(G,\frac{p}{L}\right)\right]\leq \frac{c}{L^k}.
    \end{alignat*}
\end{lemma}
\begin{proof}\ \\
By the linearity of expectation we get
    \begin{alignat*}{100}
        \mathbb{E}\left[w\left(G,\frac{p}{L}\right)\right] = \sum_{S\in \binom{n}{k}} \mathbb{P}\left(S\in G\right)\cdot \left(\frac{p}{L}\right)^k \leq \frac{1}{L^k}\cdot\sum_{S\in \binom{n}{k}} c\cdot g(S)\cdot p^k = \frac{c}{L^k}\cdot w\left(g,p\right) = \frac{c}{L^k}
    \end{alignat*}
\end{proof}

\begin{proof}[Proof of Theorem~\ref{thm:randomized}]
We set $c=n+1$ with foresight. Lemma~\ref{LemmaProbCovering} yields:
 \[
 \mathbb{P}\left(\left<g\right>\subseteq \left<G\right>\right)\geq \frac{1}{2},
 \]
whereas Lemma \ref{LemmaExpectedWeight} (applied with $L=4n^{1/k}$) guarantees 
    \begin{alignat*}{100}
        \mathbb{E}\left[w\left(G,\frac{p}{4\cdot n^{\frac{1}{k}}}\right)\right]\leq \frac{n+1}{4^k\cdot n} \leq \frac{1}{2}.
    \end{alignat*}
Consequently we obtain with $w(\cdot,\cdot)\geq 0$ and $\mathbb{E}\left[X\right]=\mathbb{P}\left(A\right)\mathbb{E}\left[X\mid A\right]+\mathbb{P}\left(\lnot A\right)\mathbb{E}\left[X\mid \lnot A\right]$ for any event $A$ and random variable $X$:
    \begin{alignat*}{100}
        \mathbb{E}\left[\left. w\left(G,\frac{p}{4\cdot n^{\frac{1}{k}}}\right)\right|\left<g\right>\subseteq \left<G\right>\right]\leq 1.
    \end{alignat*}
  Hence there has to be a choice of $G$ that fullfils $\left<g\right>\subseteq \left<G\right>$ and $w\left(G,\frac{p}{4\cdot n^{\frac{1}{k}}}\right)\leq 1$, which completes the proof of the first statement of Theorem~\ref{thm:randomized}.
  
 The second assertion of Theorem~\ref{thm:randomized} follows by using $k\geq C \log_2(n)$ and the monotonicity of $w\left(G,p\right)$ in $p$:
    \begin{alignat*}{100}
w\left(G,\frac{p}{4\cdot 2^\frac{1}{C}}\right) = w\left(G,\frac{p}{4\cdot n^\frac{1}{C\log_2 n}}\right) \le        w\left(G,\frac{p}{4\cdot n^\frac{1}{k}}\right) \le     1.
    \end{alignat*}
\end{proof}

\section{Proof of Theorem~\ref{TheoremSimpleHypergraphs}}
The proof of Theorem~\ref{TheoremSimpleHypergraphs} follows along the lines of the proof of Frankston, Kahn and Park~\cite{TalagrandTwoSets} for functions supported by $2$-element sets. However, we write it up for the statement according to Conjecture~\ref{conj:Talagrand_b} and the details are quite different, require adaptations at several places and also lead to certain technical simplifications so that we see no other way than  to provide the proof in its entirety.

As already mentioned in the introduction, Talagrand~\cite[Section 6]{TalagrandOriginal} solved the special case of Conjecture~\ref{conj:Talagrand_b} when $g$ is supported by $1$-element subsets of $X$. We provide its short proof for the sake of completeness.
\begin{proposition}
\label{TheoremSingletons}
    For every $g: \binom{X}{1} \rightarrow [0,1]$ and $p\in [0,1]$ the following holds. If 
    $w(g,p) = 1$ then there exists a set $G\subseteq 2^X$ with $\left<g\right>\subseteq \left<G\right>$ and $w\left(G,\frac{p}{4e}\right) \leq 1$.    
\end{proposition}
\begin{proof}
    Let $\left(T_i\right)_{i=1}^n$ be the elements of $\binom{X}{1}$ such that $g\left(T_i\right)\geq g\left(T_{i+1}\right)$ for all $i\in [n-1]$. 
    We define $a:=\left\lceil\sum_{i=1}^n g\left(T_i\right)\right\rceil$ and observe that $a\leq 2\cdot \sum_{i=1}^n g\left(T_i\right) =2\cdot p^{-1}$
since  $\sum_{i=1}^n g\left(T_i\right)=p^{-1}\ge 1$. We then define $G$ as follows
    \begin{alignat*}{100}
        G := \bigcup_{j=1}^n \bigcup_{I\in\binom{\left[\min\{j\cdot a, n\}\right]}{j}} \left\{\bigcup_{i\in I} T_i\right\}.
    \end{alignat*}
    
    Since  $(g\left(T_i\right))_{i\in[n]}$ is a monotone decreasing sequence, we have that
    \begin{alignat}{100}\label{eq:one-a-bound}
        a\geq \sum_{i=1}^n g\left(T_i\right)> a\cdot \sum_{j: j \cdot a < n} g\left(T_{j \cdot a + 1}\right).
    \end{alignat}
    
Assume now there is a set $S\in \left<g\right>$ such that $S\not\in \left<G\right>$. Let $I\subseteq [n]$ be such that $S=\bigcup_{i\in I} T_i$. 
Then we know that $\sum_{i\in I} g\left(T_i\right)\geq 1$ and $\left| I\cap \left[j\cdot a\right] \right|< j$ for all $j\in [n]$. From this it follows directly that
    \begin{alignat*}{100}
        a\cdot \sum_{j: j \cdot a < n} g\left(T_{j \cdot a + 1}\right) \geq a\cdot \sum_{i\in I} g\left(T_i\right) \geq a.
    \end{alignat*}
    Which is a contradiction to~\eqref{eq:one-a-bound} and therefore  $\left<g\right>\subseteq \left<G\right>$ holds. Finally we compute the weight $w\left(G,\frac{p}{4e}\right)$
     as follows
    \begin{alignat*}{100}
        w\left(G,\frac{p}{4e}\right) \leq \sum_{j=1}^n \binom{j\cdot a}{j} \left(\frac{p}{4e}\right)^j \leq \sum_{j=1}^n \left(\frac{e\cdot j\cdot a}{j}\right)^j \left(\frac{p}{4e}\right)^j\overset{a\leq 2\cdot p^{-1}}{\leq} \sum_{j=1}^n \left(\frac{e\cdot 2\cdot p^{-1}\cdot p}{4e}\right)^j = \sum_{j=1}^n \frac{1}{2^j} \leq 1
    \end{alignat*}
\end{proof}

Recall from the introduction that from  $\supp(g)\subseteq \binom{X}{k}$ and $w(g,p)=1$ we can directly compute $p$ as $p = \left(\sum_{T\in \binom{X}{k}} g\left(T\right)\right)^{-\frac{1}{k}}$. The following definition introduces the set $\lgrJL\subseteq 2^X$ which requires `higher weight' for a set $S\in\lgr$ to be included into $\lgrJL$. 

\begin{definition}\label{def:lgrJL}
For $J$, $L\ge 1$, $k\in\NN$, a finite set $X$ and a function $g\colon \binom{X}{k}\to[0,\infty)$ define 
\begin{alignat*}{100}
    \left<g\right>_{J,L} := \left\{S\in 2^X\left| \sum_{T\in \binom{S}{k}} g(T)\geq \max\left\{J, \frac{L}{4ek}\cdot \left(\sum_{T\in \binom{X}{k}} g(T)\right)^{-1+\frac{1}{k}}\cdot \sum_{x\in S}\sum_{T: x\in T} g(T)\right\}\right.\right\}.
\end{alignat*}
\end{definition}
We remark that $\left(\sum_{T\in \binom{X}{k}} g(T)\right)^{-1+\frac{1}{k}}\cdot \sum_{x\in S}\sum_{T: x\in T} g(T)$ is equal to $p^{k-1} \sum_{x\in S}\sum_{T: x\in T} g(T)$ which can be thought of as the sum of weighted vertex degrees from $S$ in the weighted $k$-uniform hypergraph $g$ and that for $k=2$ this is the weight considered in~\cite{TalagrandTwoSets}.

The following theorem below is a generalization of Theorem~2.2 from~\cite{TalagrandTwoSets} to linear hypergraphs with constant weight (i.e.\ $\supp(g)$ is linear when viewed as a $k$-uniform hypergraph and $g$ is interpreted as its weight). The main difference  in the proof is that we exploit the observation that a $k$-uniform star with $m$ edges has exactly $\left(k-1\right)\cdot m +1$ vertices and we work directly with the function $g$.

\begin{theorem}
\label{TheoremLinearHypergraphsUnweighted}

    For all $n$, $k\in\NN$ with $n\ge k$, $J\ge 1$, $r>0$, $L\geq 2^{10}\cdot e^2$, any $n$-element set $X$ and for any function  $g: \binom{X}{k} \rightarrow [0,\infty)$ which equals $\frac{1}{r}$ on its support and with  $\Delta_2(\supp(g))\le 1$ the following holds. If $w(g,p)=1$ then there exists a set $G\subseteq 2^X$ with $\lgrJL \subseteq \lGr$ and $w\left(G,\frac{p}{L}\right) \leq  \left(\frac{1}{L}\right)^{\frac{\sqrt{J\cdot r}}{2^7}}$.
\end{theorem}
\begin{proof}
We may assume w.l.o.g.\ that $k\ge 2$ as otherwise $\lgrJL=\emptyset$ and the assertion is trivial.

 We may also assume w.l.o.g.\  that  $J\cdot r>2^5\cdot k$ since otherwise we would have $\frac{2^5 k}{r}\ge J\ge 1$ and we could just take $G=\supp (g)$ because then $(2^5 k) \cdot g$ attains value at least one on its support and also $w\left((2^5 k) \cdot g,\frac{p}{L}\right)=\frac{2^5  k}{L^k}\cdot w(g,p) \leq \left(\frac{1}{L}\right)^{\frac{k}{2}}= \left(\frac{1}{L}\right)^{\frac{2^5 k}{2^6}} \leq \left(\frac{1}{L}\right)^{\frac{\sqrt{J r}}{2^6}}$. 
    Therefore we can fix an $\ell\in\NN$ such that 
\begin{equation}\label{eq:ell}
    2^{2\cdot \ell+3} < \frac{J\cdot r}{k-1} \leq 2^{2\cdot (\ell+1)+3}= 2^{2\cdot \ell+5}.
\end{equation}    
  For every $i\in [\ell]$,  let 
\begin{equation}\label{eq:L_b_i}
  L_i:= 2^{i-1} \quad \text{ and }\quad b_i:=2^{2\cdot (\ell-i)-\min\left\{i-1,\ell-i\right\}}. 
\end{equation}
  For every $x\in X$ define $N(x):=\left\{T\in\supp(g)|x\in T\right\}$ to be the \emph{neighborhood of $x$} and $\deg(x):=\left|N(x)\right|=r\cdot \sum_{T: x\in T} g(T)$ its \emph{degree} (by multiplying with $r$, $\deg(x)$ is the number of edges of $\supp(g)$ incident to $x$). Observe that 
  \begin{equation}\label{eq:degree_sum}
   \sum_{x\in X}\deg(x)=k \cdot \sum_{T\in\supp(g)} rg(T)=k r p^{-k},
  \end{equation}
  since, by assumption, $w(g,p)=1$.  
  A collection $U$ of edges from the $k$-uniform hypergraph $\supp(g)$ will be called a \emph{$k$-uniform star with center $x$} (or shortly: \emph{a star}), if all sets from $U$ intersect pairwise exactly in $x$.

    Next we define the following set of vertex-disjoint stars $G(b_i,L_i)$ in the hypergraph $\supp(g)$: \newpage
\begin{multline}\label{eq:GbiLi}
        G(b_i,L_i) := \left\{\dot\bigcup_{j=1}^{b_i}\bigcup_{E\in U_j} E \,\bigg|\, 
        \forall j\in [b_i]\,\exists  x_j\in X \text{ and } U_j\subseteq N(x_j) \text{ such that } \right.\\ \left.
       \left|U_j\right|=L_i\geq \frac{L}{8ek}\cdot \deg(x_j)\cdot \left(\sum_{T\in \binom{X}{k}} g(T)\right)^{-1+\frac{1}{k}} 
 \text{ and } \forall j\neq j': \left(\bigcup_{E\in U_j} E\right)\bigcap\left(\bigcup_{E\in U_{j'}} E\right)=\emptyset \right\}.
    \end{multline}    
Each set from $G\left(b_i,L_i\right)$ consists of exactly $b_i$ vertex-disjoint stars,  each star with $L_i$ edges, where it is also additionally required that $L_i\geq \frac{L}{8ek}\cdot \deg(x_j)\cdot \left(\sum_{T\in \binom{X}{k}} g(T)\right)^{-1+\frac{1}{k}}$ holds. The set $G\subseteq 2^X$ consists then of all possible sets from $G\left(b_i,L_i\right)$ for some $i\in[\ell]$:
\begin{equation*}
G := \bigcup_{i=1}^\ell G\left(b_i,L_i\right)
\end{equation*}
    
 In the following we will verify that $G$ covers $\lgrJL$ and that the weight of $G$ is as claimed, i.e.\ at most $\left(\frac{1}{L}\right)^{\frac{\sqrt{J\cdot r}}{2^7}}$.
\begin{claim}\label{claim:covering}
$
\lgrJL\subseteq \lGr.
$
\end{claim}
\begin{proof}[Proof of Claim~\ref{claim:covering}]
Let $S\in \left<g\right>_{J,L}$. Then we know from the definition of $\left<g\right>_{J,L}$ that
    \begin{alignat*}{100}
        \sum_{T\in\binom{ S}{k}} g(T)\geq \max\left\{J,  \frac{L}{4ek}\cdot \left(\sum_{T\in \binom{X}{k}} g(T)\right)^{-1+\frac{1}{k}}\cdot \sum_{x\in S}\sum_{T: x\in T} g(T)\right\}.
    \end{alignat*}
We will construct pairs $\left(x_j,U_j\right)$ for $j\in[n]$, where $x_j\in S$ and $U_j$ is a star with center $x_j$ in the induced hypergraph $\supp(g)[S]$. We will proceed inductively repeating the following steps (starting with $j=1$ and increasing the value of $j$ until no star can be found anymore):
    \begin{enumerate}
        \item Choose $x_j\in S\setminus \left(\bigcup_{j'<j} U_{j'}\right)$ such that the size of $U_j:=N\left(x_j\right)\cap\binom{S\setminus \left(\bigcup_{j'<j} U_{j'}\right)}{k}$ is maximised under the condition that it has size at least $\frac{L}{8 e k}\cdot \deg(x_j)\cdot \left(\sum_{T\in \binom{X}{k}} g(T)\right)^{-1+\frac{1}{k}}$.
        \item If no such $x_j$ exists we set $U_{j'}=\emptyset$ for all $j'\geq j$ and stop.
    \end{enumerate}
This procedure asserts that we end up with a collection of pairwise disjoint sets in $S$ (corresponding to vertex-disjoint stars in $\supp(g)[S]$). Next we argue that there will always be a set from  $G$ (defined in~\eqref{eq:GbiLi} for some $i\in[\ell]$)  contained in  some union of $U_j$'s which we found, and hence contained in the set $S$. 

We set $d_j:=\left|U_j\right|$ (for all $j\in[n]$) and observe that for every $E\in \supp(g)[S]$ at least one of the following is true:
    \begin{enumerate}[(i)]
        \item $E$ is contained in some $U_j$.
        \item $E$ was removed because it intersects an $E'\in U_j$ for some $j$.
        \item Some vertex $x$ from $E$ has degree less than $\frac{L}{8 e k}\cdot \deg(x)\cdot \left(\sum_{T\in \binom{X}{k}} g(T)\right)^{-1+\frac{1}{k}}$ in the induced hypergraph $\supp(g)\left[S\setminus \left(\bigcup_{j} U_{j}\right)\right]$.
    \end{enumerate}
    For the possibility (iii)  above we account at most 
\[
\sum_{x\in S} \frac{L}{8 e k}\cdot \deg(x)\cdot \left(\sum_{T\in \binom{X}{k}} g(T)\right)^{-1+\frac{1}{k}}
\]
 edges (by adding up all the degrees of vertices that can not be chosen as a center). 
 We account for the possibilities (i) and (ii), whereby excluding those edges already covered by (iii), at most $\sum_{j=1}^n \left(k-1\right)\cdot d_j^2$ edges. This is due to the fact, that a star with $d_j$ edges contains exactly $(k-1)d_j+1$ vertices (recall that $\supp(g)$ is a linear hypergraph) and the current maximum degree of an eligible vertex at step $j$ is $d_j$, which results in  removing at most $(k-1)d_j^2$ edges from $\supp(g)[S]$ plus those edges already taken care of by (iii). 
 
 The sum of these both estimates is an upper bound on the total number of edges in $\supp(g)[S]$: 
    \begin{alignat*}{100}
         \sum_{j=1}^n \left(k-1\right)\cdot d_j^2 + \sum_{x\in S} \frac{L}{8 e k}\cdot \deg(x)\cdot \left(\sum_{T\in \binom{X}{k}} g(T)\right)^{-1+\frac{1}{k}} \geq \sum_{T\in \binom{S}{k}} \mathds{1}_{\left\{T\in\supp(g)\right\}}= r\cdot\sum_{T\in \binom{S}{k}} g(T).
    \end{alignat*}
    
    With $\deg(x)=r\cdot \sum_{T: x\in T} g(T)$ and $S\in \left<g\right>_{J,L}$ (cf.\ Definition~\ref{def:lgrJL}) we know:
    \begin{alignat*}{100}
        \sum_{x\in S} \frac{L}{8 e k}\cdot \deg(x)\cdot \left(\sum_{T\in \binom{X}{k}} g(T)\right)^{-1+\frac{1}{k}} = \frac{r}{2}\cdot  \frac{L}{4ek}\cdot \sum_{x\in S}\sum_{T: x\in T} g(T)\cdot \left(\sum_{T\in \binom{X}{k}} g(T)\right)^{-1+\frac{1}{k}} \leq \frac{r}{2}\cdot \sum_{T\in \binom{S}{k}} g(T).
    \end{alignat*}
    
The above inequalities together imply:
    \begin{alignat*}{100}
        \sum_{j=1}^n \left(k-1\right)\cdot d_j^2 \geq \frac{r}{2}\cdot\sum_{T\in \binom{S}{k}} g(T) \overset{\text{Definition~\ref{def:lgrJL}}}{\geq} \frac{r\cdot J}{2} \geq 2^{2\cdot \ell+2}\cdot \left(k-1\right),
    \end{alignat*}
 where the last inequality used the choice of $\ell$ (cf.~\eqref{eq:ell}).   
    
    Now, for every $i\in[\ell]$ we define $B_i:=\left\{d_j\left|L_i\leq d_j< L_{i+1}\right.\right\}$, where we set $L_{\ell+1}:=\infty$. 
    We aim to find an $i\in[\ell]$ such that $\left|B_i\right|\geq b_i$. 
 If $\left|B_\ell\right|\geq 1=b_\ell$ then we are done. Otherwise we have
    \begin{alignat*}{100}
\sum_{i=1}^{\ell-1} \left(k-1\right)\cdot \left|B_i\right|\cdot L_{i+1}^2 \geq \sum_{j=1}^n \left(k-1\right)\cdot d_j^2\geq \left(k-1\right)\cdot 2^{2\cdot \ell+2},
    \end{alignat*}
which simplifies to (recalling the definition of $L_i=2^{i-1}$, cf.~\eqref{eq:L_b_i})
    \begin{alignat*}{100}
 \sum_{i=1}^{\ell-1} \left|B_i\right|\cdot 2^{2\cdot (i-\ell) -2} \geq 1.
    \end{alignat*}  
   
We rewrite the sum by relating it to the $b_i$s (cf.~\eqref{eq:L_b_i}):
    \begin{alignat*}{100}
        1\le \sum_{i=1}^{\ell-1} \left|B_i\right|\cdot 2^{2\cdot (i-\ell) -2} =  \sum_{i=1}^{\ell-1} \frac{\left|B_i\right|}{2^{2\cdot (\ell-i) -\min\left\{i-1,\ell-i\right\}}}\cdot \frac{2^{-\min\left\{i-1,\ell-i\right\}}}{4} \leq \frac{1}{2}\cdot \sum_{i=1}^{\ell-1} \frac{\left|B_i\right|}{b_i}\cdot \left(2^{-i}+2^{-\ell+i-1}\right),
    \end{alignat*}
    from which the existence of an $i\in[\ell-1]$ with $\left|B_i\right|\geq b_i$ follows by the geometric sum.
\end{proof}

Finally we turn to calculating the weight of $G$.
\begin{claim}\label{claim:weight}
$w\left(G,\frac{p}{L}\right) \leq \left(\frac{1}{L}\right)^{\frac{\sqrt{J\cdot r}}{2^7}}$.
\end{claim}
\begin{proof}[Proof of Claim~\ref{claim:weight}]
We will use that $w(g,p) = 1$ is equivalent to $p=\left(\sum_{T\in \binom{X}{k}} g(T)\right)^{-\frac{1}{k}}$. Hence the lower bound $\frac{L}{8ek}\cdot \deg(x_j)\cdot \left(\sum_{T\in \binom{X}{k}} g(T)\right)^{-1+\frac{1}{k}}$ on $L_i$ in the definition of $G\left(b_i,L_i\right)$ (cf.~\eqref{eq:GbiLi}) translates to $L_i\ge \frac{L}{8ek}\cdot \deg(x_j)\cdot \left(p^{-k}\right)^{-1+\frac{1}{k}}=\frac{L}{8ek}\cdot \deg(x_j)\cdot p^{k-1}$. Moreover, by construction of $G\left(b_i,L_i\right)$, every set $S\in G\left(b_i,L_i\right)$ has cardinality exactly $b_i\cdot\left(\left(k-1\right)\cdot L_i +1\right)$ (there are $b_i$ vertex disjoint $k$-uniform stars, each with $L_i$ edges). Therefore, we estimate:
    \begin{alignat*}{100}
        w\left(G,\frac{p}{L}\right) & \leq \sum_{i=1}^\ell \left|G\left(b_i,L_i\right)\right|\cdot \left(\frac{p}{L}\right)^{b_i\cdot\left(\left(k-1\right)\cdot L_i +1\right)}\\
         & \leq \sum_{i=1}^\ell \sum_{Z\in\binom{X}{b_i}} \left(\prod_{z\in Z} \binom{\deg(z)}{L_i}\cdot\one_{\left\{L_i\geq \frac{L}{8ek}\cdot \deg(z)\cdot p^{k-1}\right\}}\right)\cdot \left(\frac{p}{L}\right)^{b_i\cdot\left(\left(k-1\right)\cdot L_i +1\right)}.
    \end{alignat*}
We simplify each most inner term as follows
\begin{alignat*}{100}
\binom{\deg(z)}{L_i}\cdot\mathds{1}_{\left\{L_i\geq \frac{L}{8ek}\cdot \deg(z)\cdot p^{k-1}\right\}}\le 
\left(\frac{e \deg(z)}{L_i}\right)^{L_i} \one_{\left\{L_i\geq \frac{L}{8ek}\cdot \deg(z)\cdot p^{k-1}\right\}}\le \\
\\
\left(\frac{e \deg(z)}{L_i}\right) \cdot \left(\frac{e \deg(z)} { \frac{L} {8ek}\cdot \deg(z)\cdot p^{k-1}}\right)^{L_i-1}=
\left(\frac{e \deg(z)}{L_i}\right) \cdot \left(\frac{8 e^2 k}{ L p^{k-1}}\right)^{L_i-1}.
\end{alignat*}
We obtain:
\begin{alignat*}{100}
        w\left(G,\frac{p}{L}\right)\le \sum_{i=1}^\ell \sum_{Z\in\binom{X}{b_i}} \left(\prod_{z\in Z} \left(\frac{e \deg(z)}{L_i}\right) \cdot \left(\frac{8 e^2 k}{ L p^{k-1}}\right)^{L_i-1}\right)\cdot \left(\frac{p}{L}\right)^{b_i\cdot\left(\left(k-1\right)\cdot L_i +1\right)}\\
        = \sum_{i=1}^\ell \left(
         \sum_{Z\in\binom{X}{b_i}} \left(\prod_{z\in Z} \deg(z)\right)
         \right) \cdot \left(\frac{e}{L_i} \right)^{b_i}\cdot \left(\frac{8 e^2 k}{ L p^{k-1}}\right)^{b_i(L_i-1)}\cdot \left(\frac{p}{L}\right)^{b_i\cdot\left(\left(k-1\right)\cdot L_i +1\right)}.
\end{alignat*} 
Now observe that the right hand side above is maximised when all $\deg(z)$ are the same for all $z$. 
    In this case we have  for all $z\in X$ that $\deg(z)=\frac{r\cdot p^{-k}\cdot k}{n}$ (cf.\eqref{eq:degree_sum}). And so we get:
     \begin{alignat}{100}
        w\left(G,\frac{p}{L}\right) &\le &&  
       \sum_{i=1}^\ell \binom{n}{b_i} \left(\frac{r p^{-k} k}{n}\right)^{b_i} 
 \cdot \left(\frac{e}{L_i} \right)^{b_i}\cdot \left(\frac{8 e^2 k}{ L p^{k-1}}\right)^{b_i(L_i-1)}\cdot \left(\frac{p}{L}\right)^{b_i\cdot\left(\left(k-1\right)\cdot L_i +1\right)} \notag\\
  &\le && \sum_{i=1}^\ell \left(\frac{en}{b_i}\right)^{b_i} \left(\frac{r p^{-k} k}{n}\right)^{b_i}
 \cdot \left(\frac{e}{L_i} \right)^{b_i}\cdot \left(\frac{8 e^2 k}{ L p^{k-1}}\right)^{b_i(L_i-1)}\cdot \left(\frac{p}{L}\right)^{b_i k}\cdot \left(\frac{p}{L}\right)^{\left(k-1\right) b_i  \left(L_i-1\right)} \notag\\
 &= && \sum_{i=1}^\ell \left(\frac{en\cdot r p^{-k} k\cdot e\cdot p^k}{b_i\cdot n\cdot L_i\cdot L^k}\right)^{b_i} 
 \cdot \left(\frac{8 e^2 k \cdot p^{k-1}}{ L p^{k-1}\cdot L^{k-1}}\right)^{b_i(L_i-1)} \notag\\
         & = && \sum_{i=1}^\ell \left(\frac{e^2 r k}{b_i  L_i  L^k}\right)^{b_i}\cdot \left(\frac{8 e^2  k}{L^k}\right)^{b_i(L_i-1)}\notag\\
         & \overset{\mathclap{\eqref{eq:ell},\eqref{eq:L_b_i}}}{\leq} && \quad \sum_{i=1}^\ell \left(\frac{e^2 2^{2\cdot \ell+5} k^2}{2^{2 (\ell-i)-\min\left\{i-1,\ell-i\right\}}\cdot 2^{i-1}\cdot L^k\cdot J}\right)^{b_i}\cdot \left(\frac{8 e^2 k}{L^k}\right)^{b_i(L_i-1)}\notag\\
         & = && \sum_{i=1}^\ell \left(\frac{e^2  k^2}{L^k J}\cdot 2^{i+6+\min\left\{i-1,\ell-i\right\}}\right)^{b_i}\cdot \left(\frac{8 e^2 k}{L^k}\right)^{b_i(L_i-1)} 
         \leq \sum_{i=1}^\ell \left(\frac{e^2 2^7 k^2}{L^k J}\cdot 2^{2i-2}\right)^{b_i}\cdot \left(\frac{2^5 e^2 k}{4L^k}\right)^{b_i(L_i-1)} 
         \notag\\
         &\overset{\mathclap{\eqref{eq:L_b_i}}}{=} \ &&     \sum_{i=1}^\ell \left(\frac{e^2 2^7 k^2}{L^k J}\cdot L_i^2\cdot 4^{-L_i+1}\right)^{b_i}\cdot \left(\frac{2^5 e^2 k}{L^k}\right)^{b_i(L_i-1)} 
 \overset{\eqref{eq:L_b_i}}{\leq} \sum_{i=1}^\ell \left(\frac{e^2 2^7 k^2}{L^k  J}\right)^{b_i}\cdot \left(\frac{2^5  e^2  k}{L^k}\right)^{b_i(L_i-1)}.\label{eq:GpL}
    \end{alignat}
By assumption of the theorem, $L\geq 2^{10}\cdot e^2$ and since $k\ge2$ we have $\frac{e^2 2^7 k^2 }{L^{k/2} J}\le \frac{1}{2}$ and $2^5e^2k\le L^{k/2}$. 
Hence, we simplify the last term of~\eqref{eq:GpL} as follows:
\begin{alignat*}{100}
    w\left(G,\frac{p}{L}\right) & \leq && \sum_{i=1}^\ell \left(\frac{1}{2L^{\frac{k}{2}}}\right)^{b_i}\cdot \left(\frac{1}{L^{\frac{k}{2}}}\right)^{b_i\left(L_i-1\right)} = \sum_{i=1}^\ell \left(\frac{1}{2}\right)^{b_i}\cdot \left(\frac{1}{L^{\frac{k}{2}}}\right)^{b_i\cdot L_i}\\
    & \overset{\mathclap{\eqref{eq:L_b_i}}}{=} \ && \sum_{i=1}^\ell \left(\frac{1}{2}\right)^{2^{2\cdot (\ell-i)-\min\left\{i-1,\ell-i\right\}}}\cdot \left(\frac{1}{L^{\frac{k}{2}}}\right)^{2^{2\cdot (\ell-i)-\min\left\{i-1,\ell-i\right\}}\cdot 2^{i-1}}\\
    & \leq && \sum_{i=1}^\ell \left(\frac{1}{2}\right)^{2^{2\cdot (\ell-i)-\left(\ell-i\right)}}\cdot \left(\frac{1}{L^{\frac{k}{2}}}\right)^{2^{2\cdot (\ell-i)-\left(\ell-i\right)}\cdot 2^{i-1}}\\
    & = && \sum_{i=1}^\ell \left(\frac{1}{2}\right)^{2^{\ell-i}}\cdot \left(\frac{1}{L^{\frac{k}{2}}}\right)^{2^{\ell-1}} = \sum_{j=1}^\ell \left(\frac{1}{2}\right)^{2^{j-1}}\cdot \left(\frac{1}{L^{\frac{k}{2}}}\right)^{\sqrt{2^{2\ell+5-7}}}\\
    & \overset{\mathclap{\eqref{eq:ell}}}{\leq} \ && \sum_{j=1}^\ell \left(\frac{1}{2}\right)^{j}\cdot \left(\frac{1}{L^{\frac{k}{2}}}\right)^{\sqrt{\frac{J\cdot r}{2^7\cdot\left(k-1\right)}}} \leq \left(\frac{1}{L}\right)^{\frac{\sqrt{J\cdot r}}{2^4}} \leq \left(\frac{1}{L}\right)^{\frac{\sqrt{J\cdot r}}{2^7}}
\end{alignat*}
\end{proof}
This finishes the proof of Theorem~\ref{TheoremLinearHypergraphsUnweighted}.
\end{proof}

The next theorem will allow us to reduce weighted case to an unweighted version  where Theorem~\ref{TheoremLinearHypergraphsUnweighted} becomes applicable.

\begin{theorem}
\label{TheoremUnWeightedToWeighted}
For $n$, $k\in \NN$ with $n\ge k$, $J\geq 1$, any $n$-element set $X$ and any function $g\colon \binom{X}{k} \rightarrow [0,1]$ with $\sum_{S\in\binom{X}{k}}g(S)>0$ define for every $i\in\NN$
    \begin{enumerate}[(a)]
      \item 
      \[g_i(S) := \frac{1}{4^{i-1}}\cdot \mathds{1}_{\left\{\frac{1}{4^{i-1}}\geq g(S) > \frac{1}{4^{i}}\right\}}\]
      \item 
      \[
      \ell_i := \frac{\sum_{S\in\binom{X}{k}}g_i(S)}{\sum_{S\in\binom{X}{k}}g(S)} \quad\text{and for $\ell_i\neq0 \colon$}\quad
      J_i := \max\left\{1,\frac{\ell_i^{-1}}{2^{i-1}}\right\}.
      \]
    \end{enumerate}
  Suppose there exist $L$, $c\ge 1$ and $p'\in[0,1]$ such that for every $i\in \NN$ with $\ell_i\neq0$ if we have 
  $w\left(10\ell_i^{-1}\cdot g_i ,p'\right)=1$ then  there exists $G_i \subseteq 2^X$ with $\left< 10 \ell_i^{-1}\cdot g_i \right>_{J_i,\frac{L}{10}} \subseteq \left<G_i\right>$  and $w\left(G_i,\frac{p'}{L}\right) \leq \frac{c}{i^2}$. 
    
Then the following holds. If  $w(g,p) = 1$ for some  $p\in[0,1]$ then there exists 
$G\subseteq 2^X$ with $\left<g\right>_{1,L} \subseteq \left<G\right>$ and $w\left(G,\frac{p}{100\cdot c\cdot L}\right) \leq 1$.
\end{theorem}
\begin{proof} 

 Let $p\in[0,1]$ and let $g\colon \binom{X}{k}\to[0,1]$ be as in the assumption of the theorem, i.e.\ $w(g,p) = 1$.

From the definition of $g_i$ and $\ell_i$, we have for every $i\in \NN$ with $\ell_i\neq0$ that  $\sum_{S\in\binom{X}{k}}\ell_i^{-1}\cdot g_i(S)=\sum_{S\in\binom{X}{k}}g(S)$. Since $w(g,p)=1$ we can choose $p':=\frac{p}{10^{1/k}}$ and obtain $w\left(10\ell_i^{-1}\cdot g_i,p'\right)=1$. Therefore, by the assumption of the theorem, for every $i$ with $\ell_i\neq0$, there exists $G_i \subseteq 2^X$ with $\left< 10 \ell_i^{-1}\cdot g_i \right>_{J_i,\frac{L}{10}} \subseteq \left<G_i\right>$  and $w\left(G_i,\frac{p'}{L}\right) \leq \frac{c}{i^2}$.  We set 
 \[
 G=\bigcup_{i=1\colon \ell_i\neq 0}^\infty G_i
 \]
 and we will verify the assertion of the theorem for $G$.

First we estimate (using monotonicity of the weight function):
    \begin{alignat*}{100}
       w\left(G,\frac{p}{100\cdot c\cdot L}\right) \leq \frac{1}{10 c}\cdot \sum_{i=1 \colon \ell_i\neq 0}^{\infty} w\left(G_i,\frac{10\cdot p'}{10\cdot L}\right) \leq \frac{1}{10 c}\cdot \sum_{i=1}^{\infty} \frac{c}{i^2} \leq 1
    \end{alignat*}

Then we turn to show $\left<g\right>_{1,L} \subseteq \left<G\right>$. Take any $S\in \left<g\right>_{1,L}$. By the definition of $\lgr_{1,L}$ (cf.\ Definition~\ref{def:lgrJL}),
we have 
\begin{equation}\label{eq:g1L}
\sum_{T\in \binom{S}{k}}g(T)\geq \max
\left\{1, 
\frac{L}{4ek}\cdot \left(\sum_{T\in \binom{X}{k}} g(T)\right)^{-1+\frac{1}{k}}\cdot \sum_{x\in S}\sum_{T: x\in T} g(T)
\right\}.
\end{equation}
Moreover, from the definition of the functions $g_i$ and parameters $\ell_i$ we have $g(T)\le \sum_{i=1}^\infty g_i(T)\leq 4\cdot g(T)$ for all $T$ and $\sum_{i=1}^\infty \ell_i \leq 4$. 
We thus obtain from~\eqref{eq:g1L} the following inequality 
\begin{alignat*}{100}
10\cdot \sum_{T\in \binom{S}{k}} \sum_{i=1}^\infty g_i(T) \geq& \sum_{i=1}^\infty \ell_i + \sum_{i=1}^\infty \frac{1}{2^{i-1}} + \frac{L}{4ek}\cdot \left(\sum_{T\in \binom{X}{k}} g(T)\right)^{-1+\frac{1}{k}}\cdot \sum_{x\in S}\sum_{T: x\in T} \sum_{i=1}^\infty g_i(T)\\
=& \sum_{i=1}^\infty\left(\ell_i+\frac{1}{2^{i-1}}+\frac{L}{4ek}\cdot \left(\sum_{T\in \binom{X}{k}} g(T)\right)^{-1+\frac{1}{k}}\cdot \sum_{x\in S}\sum_{T: x\in T} g_i(T)\right).
\end{alignat*}
It follows that there exists an index $i\in \NN$ with $\ell_i\neq 0$ such that 
\begin{alignat*}{100}
10\cdot \sum_{T\in \binom{S}{k}}  g_i(T) \geq\ell_i+\frac{1}{2^{i-1}}+\frac{L}{4ek}\cdot \left(\sum_{T\in \binom{X}{k}} g(T)\right)^{-1+\frac{1}{k}}\cdot \sum_{x\in S}\sum_{T: x\in T} g_i(T)
\end{alignat*}
and therefore (since $\sum_{T\in\binom{X}{k}}\ell_i^{-1}\cdot g_i(T)=\sum_{T\in\binom{X}{k}}g(T)=p^{-k}\ge 1$) we get
\begin{alignat*}{100}
 \sum_{T\in \binom{S}{k}} 10\cdot \ell_i^{-1}\cdot g_i(T)  &\geq  1+\frac{\ell_i^{-1}}{2^{i-1}}+ \frac{L}{4ek}\cdot \left(\sum_{T\in \binom{X}{k}} \ell_i^{-1}\cdot g_i(T)\right)^{-1+\frac{1}{k}}\cdot \sum_{x\in S}\sum_{T: x\in T} \ell_i^{-1}\cdot g_i(T)\\
& \ge 1+\frac{\ell_i^{-1}}{2^{i-1}}+\frac{\frac{L}{10}}{4ek}\cdot \left(\sum_{T\in \binom{X}{k}} 10\cdot \ell_i^{-1}\cdot g_i(T)\right)^{-1+\frac{1}{k}}\cdot \sum_{x\in S}\sum_{T: x\in T} 10\cdot \ell_i^{-1}\cdot g_i(T).
\end{alignat*}
Recalling the definitions of $J_i$  we thus obtain:
\[
\sum_{T\in \binom{S}{k}} 10\cdot \ell_i^{-1}\cdot g_i(T) \geq \max\left\{J_i, \frac{\frac{L}{10}}{4ek}\cdot \left(\sum_{T\in \binom{X}{k}} 10\cdot \ell_i^{-1}\cdot g_i(T)\right)^{-1+\frac{1}{k}}\cdot \sum_{x\in S}\sum_{T: x\in T} 10\cdot \ell_i^{-1}\cdot g_i(T)\right\},
\]
which implies that (recalling Definition~\ref{def:lgrJL})
\[
S\in \left<10 \ell_i^{-1}\cdot g_i\right>_{J_i,\frac{L}{10}}.
\]
Therefore this yields  $S\in \bigcup_{i=1\colon \ell_i\neq0}^\infty \left<10\ell_i^{-1}\cdot g_i\right>_{J_i,\frac{L}{10}}$ which finishes the proof due to our choice of $G=\bigcup_{i=1\colon \ell_i\neq 0}^\infty G_i$ such that $\left< 10 \ell_i^{-1}\cdot g_i \right>_{J_i,\frac{L}{10}} \subseteq \left<G_i\right>$.
\end{proof}
 
 Now we are in position to provide the proof of Theorem~\ref{TheoremSimpleHypergraphs}. 
\begin{proof}[Proof of Theorem~\ref{TheoremSimpleHypergraphs}]
    First we consider the special case when $c=1$, hence $\supp(g)$ is a linear $k$-uniform hypergraph, i.e.\
    \begin{alignat*}{100}
        \left|\left\{T\in \supp(g)\left|x,y\in T\right.\right\}\right|&\leq 1
    \end{alignat*}
 We set  $L=10\cdot 2^{10}\cdot e^2$  with foresight and consider the set
    \begin{alignat*}{100}
        V_L := \left\{S\in 2^X \,\left|\, \frac{L}{4ek}\cdot \left(\sum_{T\in \binom{X}{k}} g(T)\right)^{-1+\frac{1}{k}}\cdot \sum_{x\in S}\sum_{T: x\in T} g(T)\geq 1\right.\right\}.
    \end{alignat*}
Observe that $\left<g\right>\subseteq \left<g\right>_{1,L}\cup V_L$ since for every $S\in \left<g\right>$ we have $\sum_{T\in \binom{S}{k}} g(T) \geq 1$ and $S\not\in\lgr_{1,L}$ implies that the maximum in the definition of $\lgr_{1,L}$ is attained through the second term and hence $S\in V_L$. 

Next we consider the function $f\colon\binom{X}{1}\rightarrow [0,1], f\left(\left\{x\right\}\right):=\frac{L}{4ek}\cdot \left(\sum_{T\in \binom{X}{k}} g(T)\right)^{-1+\frac{1}{k}} \sum_{T: x\in T}g(T)$ defined for every $x\in X$. We set $\tf:=\min\{f,1\}$ and observe that $\left<\tf\right>=V_L$. We estimate 
    \begin{alignat*}{100}
         \sum_{x\in X} \tf\left(\left\{x\right\}\right) & = \sum_{x\in X}\min\left\{1, \frac{L}{4ek}\cdot \left(\sum_{T\in \binom{X}{k}} g(T)\right)^{-1+\frac{1}{k}}\cdot \sum_{T: x\in T} g(T)\right\}\\
          & \le \frac{L}{4e}\cdot \left(\sum_{T\in \binom{X}{k}} g(T)\right)^{-1+\frac{1}{k}}\cdot \sum_{T\in \binom{X}{k}} g(T)\\
          & = \frac{L}{4e}\cdot \left(\sum_{T\in \binom{X}{k}} g(T)\right)^{\frac{1}{k}}.
    \end{alignat*}
    
From the assumption $w(g,p)=1$ we have that $p=\left(\sum_{T\in \binom{X}{k}} g(T)\right)^{-\frac{1}{k}}$  and therefore 
$w\left(f,\frac{4e}{L}\cdot p\right)=1$. 
Hence, Proposition~\ref{TheoremSingletons}  asserts the existence of $G'\subseteq 2^X$ such that $\left<\tf\right>=V_L\subseteq\left< G'\right>$ and  $w\left(G',\frac{p}{L}\right)\leq 1$. 

Thus, it remains to cover $\left<g\right>_{1,L}$. For every $i\in\NN$, we define $g_i$ as follows for all $S\in\binom{X}{k}$:
\[
g_i(S) := \frac{1}{4^{i-1}}\cdot \mathds{1}_{\left\{\frac{1}{4^{i-1}}\geq g(S) > \frac{1}{4^{i}}\right\}}.
\]
Consequently, for every $i\in\NN$,  define $\ell_i$ and $J_i$ through
\[
\ell_i := \frac{\sum_{S\in\binom{X}{k}}g_i(S)}{\sum_{S\in\binom{X}{k}}g(S)} \quad\text{and for $\ell_i\neq0\colon$}\quad
      J_i := \max\left\{1,\frac{\ell_i^{-1}}{2^{i-1}}\right\}.
\]
For every $i$ with $\ell_i\neq0$ and for every $S\in\binom{X}{k}$ we set $\tg_i(S):=10\cdot \ell_i^{-1}\cdot g_i(S)$ and, by the choice of $\ell_i$, we get 
$\sum_{S\in\binom{X}{k}} \tg_i(S)=\sum_{S\in\binom{X}{k}}10\cdot g(S)$. Since $w(g,p)=1$ we have with $p':=\frac{p}{10^{1/k}}$ that $w\left(\tg_i,p'\right)=1$ for all $i$ (where $\ell_i\neq 0$).  Recall that  $\tg_i$ is constant on its support $\left\{\frac{1}{4^{i-1}}\geq g(S) > \frac{1}{4^{i}}\right\}$ and equals to $\frac{1}{r_i}$ with
\begin{equation}\label{eq:r_i}
r_i:=\frac{\ell_i\cdot 4^{i-1}}{10}.
\end{equation}

For any $i\in\NN$ with $\ell_i\neq 0$, the assumptions of Theorem~\ref{TheoremLinearHypergraphsUnweighted} are fulfilled 
(with $\tg_i$ as $g$, $p'$ as $p$,  $L/10$ as $L$, $J_i$ as $J$ and $r_i$ as $r$).  Therefore, we find a set $G_i\subseteq 2^X$ with  
$\left<\tg_i\right>_{J_i,\frac{L}{10}}\subseteq \left<G_i\right>$ and weight
 \begin{equation}\label{eq:weight-G_i}
 w\left(G_i,\frac{10 p'}{L}\right) \leq  \left(\frac{10}{L}\right)^{\frac{\sqrt{J_i\cdot r_i}}{2^7}}.
 \end{equation}
Next we claim that we always have $J_i\cdot r_i\ge \max\{1,\frac{2^{i-1}}{10}\}$. Indeed,   
 if  $J_i=1$ then $1\ge \frac{\ell_i^{-1}}{2^{i-1}}$ and $\ell_i\cdot 2^{i-1}\ge 1$, which implies $J_i\cdot r_i\ge \frac{2^{i-1}}{10}$ (due to~\eqref{eq:r_i}). 
If $J_i=\frac{\ell_i^{-1}}{2^{i-1}}$ then $J_i\cdot r_i=\frac{2^{i-1}}{10}$. We thus futher simplify~\eqref{eq:weight-G_i} to
\[
w\left(G_i,\frac{p'}{L}\right) \leq w\left(G_i,\frac{10 p'}{L}\right) \leq \left(\frac{10}{L}\right)^{ \sqrt{2^{i-19}}}\le \frac{200}{i^2}.
\]

Thus we have shown 
that, for every $i\in\NN$ with $\ell_i\neq 0$, we have
\[
w\left(10 \ell_i^{-1}\cdot g_i,p'\right)=1
\]
 and there exists $G_i \subseteq 2^X$ with $\left<10 \ell_i^{-1}\cdot g_i\right>_{J_i,\frac{L}{10}} \subseteq \left<G_i\right>$  and $w\left(G_i,\frac{p'}{L}\right) \leq \frac{200}{i^2}$.
  
  We thus conclude, by Theorem~\ref{TheoremUnWeightedToWeighted} (with $10\cdot 2^{10}e^2$ as $L$, $200$ as $c$), that there exists a set $G\subseteq 2^X\setminus\{\emptyset\}$ with $\left<g\right>_{1,L} \subseteq \left<G\right>$ and $w\left(G,\frac{p}{2\cdot 10^4\cdot L}\right) \leq 1$. Altogether we have
 \[
 \lgr\subseteq \left<G'\cup G\right> \quad\text{and}\quad w\left(G'\cup G,\frac{p}{4\cdot 10^4\cdot L}\right)\le 1.
 \]
  This finishes the special case when the $k$-uniform hypergraph $\supp(g)$ is linear with the constant 
\begin{equation}\label{eq:Ctilde}  
  \tC=10^5\cdot 2^{12}\cdot e^2.
\end{equation}

Next we turn to the case
    \begin{alignat*}{100}
        \Delta_2\left(\supp(g)\right)&\leq c^k.
    \end{alignat*}
Observe that we can write $g=\sum_{j=1}^{(2c)^k}g_j$ with $g_j\colon \binom{X}{k}\to[0,1]$ (so that $\sum_{S\in\binom{X}{k}} g_j(S)>0$), where $\supp(g_j)\cap\supp(g_i)=\emptyset$ for all $i\neq j\in[c^k]$ and $\supp(g_j)$ is a linear $k$-uniform hypergraph for each $j\in[c^k]$. This can be seen by considering an auxiliary graph $F$ on the vertex set $\supp(g)$, where $\{e,f\}\in \binom{\supp(g)}{2}$ is an edge whenever $|e\cap f|\ge 2$. It is clear that the maximum degree $\Delta(F)$ of $F$ is at most $\binom{k}{2}c^k<(2c)^k=:m$ and therefore the chromatic number $\chi(F)$ of $F$ is at most $m$. Each color  class $F_j\subseteq \supp(g)$ is a linear $k$-uniform hypergraph. Setting $g_j=g\cdot \one_{F_j}$ we obtain the desired decomposition. 

Next we observe that $\lgr\subseteq \cup_{j=1}^{m} \left<m\cdot g_j\right>$. Indeed, let $S\in\lgr$, then we have $\sum_{T\subseteq S} \sum_{j=1}^m g_j(T)= \sum_{T\subseteq S}g(T)\ge 1$, from which we find an index $j\in[m]$ such that $\sum_{T\subseteq S} g_j(T)\ge \frac{1}{m}$, hence $\sum_{T\subseteq S} m\cdot g_j(T)\ge 1$ and $S\in\left<m\cdot g_j\right>$. 

Next, for every $j\in[m]$, we consider the function $\left(m\cdot g_j\right)^{(m)}\colon \binom{X^{(m)}}{k}\to[0,\infty)$, which is obtained by creating $m$ independent copies of the function $m\cdot g_j$ (cf.\ Definition~\ref{def:Xm} where here we use $m\cdot g_j$ instead of $g$). 
We observe that 
\begin{alignat*}{100}
w\left(\left(m\cdot g_j\right)^{(m)},\frac{p}{(2c)^2}\right)=\frac{m}{(2c)^k}\cdot w\left(m\cdot g_j,\frac{p}{2c}\right)=\frac{m^2}{(2c)^{2k}} w\left(g_j,p\right)=w\left(g_j,p\right)\le w\left(g,p\right)=1.
\end{alignat*}
Set $h(S):=\min\{1,\left(m\cdot g_j\right)^{(m)}(S)\}$ for all $S\in\binom{X^{(m)}}{k}$ and observe that $\left<h\right>=\left<\left(m\cdot g_j\right)^{(m)}\right>$ holds. 
Moreover, we have $w\left(h,\hat{p}\right)=1$ for some $\hat{p}\ge \frac{p}{(2c)^2}$.  
Since $\supp\left(h\right)$ is a linear $k$-uniform hypergraph, we infer by the first part of the theorem that there exists a constant $\tC$ (cf.~\eqref{eq:Ctilde}) so that there exists a set $G^{(m)}_j\subseteq 2^{X^{(m)}}$ with 
$\left<h\right>\subseteq \left<G^{(m)}_j\right>$ and 
\[
w\left(G^{(m)}_j,\frac{\hat{p}}{\tC}\right)\le 1.
\]
We thus have $\left<\left(m\cdot g_j\right)^{(m)}\right>\subseteq \left<G^{(m)}_j\right>$, $w\left(G^{(m)}_j,\frac{p}{(2c)^2\cdot \tC}\right)\le 1$ and, additionally, $w\left(\left(m\cdot g_j\right)^{(m)},\frac{p}{(2c)^2}\right)\le 1$.
Lemma~\ref{lem:PowerTrick} (we apply it to  $m\cdot g_j$ for $g$,  $p/(2c)$ instead of $p$, $\tC$ instead of $L$ and $2c$ instead of $c$) asserts the existence of a set $G_j\subseteq 2^X$ with $\left<m\cdot g_j\right>\subseteq \left< G_j\right>$ and $w\left(G_j,\frac{p}{(2c)^2\cdot \tC}\right)\le \frac{1}{(2c)^k}$. 

Finally,  we set $G:=\cup_{j=1}^{m} G_j$ and, hence, $\lgr\subseteq \bigcup_{j=1}^{m} \left<m\cdot g_j\right>\subseteq \lGr$ and 
\[
w\left(G,\frac{p}{(2c)^2\cdot \tC}\right)\le 1.
\]
We thus can choose $C:=4\tC$. This finishes the proof of Theorem~\ref{TheoremSimpleHypergraphs}.
\end{proof}

\end{document}